\documentclass[a4paper, oneside, reqno, 12pt]{amsart}
\usepackage[utf8]{inputenc}
\usepackage{amsmath, amssymb, amsfonts}
\usepackage{fullpage}
\usepackage{textcomp, cmap, comment}
\usepackage{mathtools} 
\usepackage[nobysame, initials]{amsrefs}

\usepackage{comment}
\usepackage{color, xcolor}
\usepackage{graphicx}
\usepackage{tikz}
\usepackage{tikz,tkz-base}
\usepackage{pgf,tikz, tkz-euclide}
\usetikzlibrary{calc}
\usepackage{tkz-euclide}
\usepackage{pgfplots}
\usepackage{relsize}
\usepackage[shellescape]{gmp}

\usepackage{subcaption}
\theoremstyle{plain}
\newtheorem{theorem}{Theorem}

\newtheorem{lemma}[theorem]{Lemma}
\newtheorem{corollary}[theorem]{Corollary}

\newtheorem{keyobservation}[theorem]{Observation}
\theoremstyle{definition}
\newtheorem{problem}[theorem]{Problem}

\theoremstyle{remark}
\newtheorem{remark}[theorem]{Remark}

\DeclareMathOperator{\dist}{dist}
\DeclareMathOperator{\conv}{conv}

\newcommand{\Ball}{B}
\newcommand{\BallK}{K} 

\usepackage{pgf,tikz, mathrsfs}
\usetikzlibrary{arrows}

\usepackage{hyperref}
\hypersetup{
    colorlinks   = true, 
    urlcolor     = blue, 
    linkcolor    = red, 
    citecolor   = green
}
\title{Intersecting diametral balls induced by\\a geometric graph}

\author[O.~Pirahmad, A.~Polyanskii, A.~Vasilevskii]{{Olimjoni~Pirahmad, Alexandr~Polyanskii, Alexey~Vasilevskii}}
\address{Olimjoni Pirahmad,
\newline\hphantom{iii} Visual Computing Center, KAUST, Thuwal 23955-6900, Saudi Arabia
}
\email{\href{mailto:pirahmad.olimjoni@kaust.edu.sa}{pirahmad.olimjoni@kaust.edu.sa
}}

\address{Alexandr Polyanskii,
\newline\hphantom{iii} Moscow Institute of Physics and Technology, Institutskiy per. 9, Dolgoprudny, Russia 141700
}
\email{\href{mailto:alexander.polyanskii@yandex.ru}{alexander.polyanskii@yandex.ru}}
\urladdr{\url{http://polyanskii.com}}

\address{Alexey Vasilevskii,
\newline\hphantom{iii} Moscow Institute of Physics and Technology, Institutskiy per. 9, Dolgoprudny, Russia 141700
}
\email{\href{mailto:lesha.vasilevski@mail.ru}{lesha.vasilevski@mail.ru}}


\thanks{The research of the second and third authors was funded by the grant of Russian Science Federation \textnumero21-71-10092, \url{https://rscf.ru/project/21-71-10092/}. The second author is a Young Russian Mathematics award winner and would like to thank its sponsors and jury.} 

\begin{document}

\thispagestyle{empty}

\begin{abstract}
For a graph whose vertex set is a finite set  of points in the Euclidean \(d\)-space consider the closed (open) balls with diameters induced by its edges. The graph is called \textit{a (an open) Tverberg graph} if these closed (open) balls intersect. Using the idea of halving lines, we show that (\textit{i}) for any finite set of points in the plane, there exists a Hamiltonian cycle that is a  Tverberg graph; (\textit{ii}) for any \( n \) red and \( n \) blue points in the plane, there exists a perfect red-blue matching that is a Tverberg graph. Also, we prove that (\textit{iii}) for any even set of points in the Euclidean \(d\)-space, there exists a perfect matching that is an open Tverberg graph; (\textit{iv}) for any \( n \) red and \( n \) blue points in the Euclidean \(d\)-space, there exists a perfect red-blue matching that is a Tverberg graph.
\end{abstract}

\maketitle
\section{Introduction}

Tverberg's theorem is one of the essential results of modern discrete and convex geometry proved by Helge Tverberg~\cite{Tverberg1966} in 1966. It claims that for any set of \( (r-1)(d+1)+1 \) points in \( \mathbb R^d \), there exists a partition with \( r \) parts whose convex hulls intersect.

In the current paper, we consider a variation of Tverberg's problem introduced recently in \cites{bereg2019maximum, huemer2019matching, soberon2020tverberg}. 
For two points \( x, y\in \mathbb{R}^d \), we denote by \( \Ball (xy) \) the closed Euclidean ball for which the segment \( xy \) is its diameter. Let \( G \) be a graph whose vertex set is a finite set of points in \( \mathbb{R}^d \). We say that \( G \) is a \textit{Tverberg graph} if
\[
    \bigcap_{ xy\in E(G)} \Ball (xy) \neq \emptyset.
\]
Replacing \textit{closed} balls by \textit{open} balls in the definition of Tverberg graph, we define an open Tverberg graph. A graph whose vertices are points in \(\mathbb R^d\) is called an \textit{open Tverberg graph} if the open balls with diameters induced by its edges intersect. Notice that there is a Tverberg graph which is not an open Tverberg graph, that is, the intersection of the closed balls induced by its edges is a single point lying on the boundary of some of them. For example, any Hamiltonian cycle for the set of vertices of a square is such a graph. For the sake of brevity, a perfect matching for an even set of points in \( \mathbb R^d \) is called \textit{a (an open) Tverberg matching} if it is a (an open) Tverberg graph. Analogously, a Hamiltonian cycle for a set of points in \( \mathbb R^d \) is called a \textit{Tverberg cycle} if it is a Tverberg graph.

In 2019, Huemer, P\'erez-Lantero, Seara, and Silveira~\cite{huemer2019matching} showed that for any set of \( n \) red points and \( n \) blue points in the plane, there is a red-blue Tverberg matching (every edge connects points of different colors). This result can be considered as a colorful variation of the problem. Moreover, they proved that the matching maximizing the sum of the squared distances between the matched points is a Tverberg red-blue matching; see also Remark~\ref{remark:function Q in Huemer paper}. 

Later, Bereg, Chac\'on-Rivera, Flores-Pe\~naloza, Huemer, and P\'erez-Lantero~\cite{bereg2019maximum} found a second proof of the monochromatic version of the result from~\cite{huemer2019matching}, that is, for any \( 2n \) points in the plane, there is a Tverberg matching. Also, they showed that the matching maximizing the sum of the distances between the matched points is a desired matching.

Recently, Sober\'on and Tang~\cite{soberon2020tverberg} showed the existence of a Tverberg cycle for an \textit{odd} set of points in the plane. As a corollary of this theorem, they proved that for any \textit{even} set of points in the plane, there is a Hamiltonian path that is a Tverberg graph. Since a Hamitonian path for an even set of vertices contains a perfect matching, this corollary implies the result from~\cite{huemer2019matching}. Also, Sober\'on and Tang initiated the study of Tverberg graphs in higher dimensions, and in particular, they considered the problem of describing the family of Tverberg graphs for a finite set of points in \( \mathbb R^d \); see Problem~1.1 in~\cite{soberon2020tverberg}. 

In 2007, twelve years before the paper~\cite{bereg2019maximum}, Dumitrescu, Pach, and T\'oth proved a lemma that is interesting in the context of Tverberg matchings; see Lemma~2 from~\cite{dumitrescu2009drawing}. This lemma easily implies that for any \( 2n \) distinct points in the plane, there exist a perfect matching \( \mathcal M \) and a point \( z \) in the plane such that either \( z \in \{x,y\} \) or \( \angle xzy \geq 2\pi/3 \) for all \( xy\in \mathcal M\). This result can be viewed as a strengthening of the existence of a Tverberg matching for an even set of points in the plane. Namely, this result claims that the common point of disks induced by the edges of the matching lies relatively deep inside of each disk. We refer to Subsection~\ref{section:discussion-intersection-of-lenses}, where we discuss a higher-dimensional version of this problem.

The goal of the current paper is to show a variety of methods that can be useful in proving the existence of Tverberg cycles or matchings for point sets.

Using the idea of halving lines~\cite{lovasz1971number}, we give short proofs of the following results.
\begin{theorem}
\label{theorem:cycleintheplane}
For any finite set of points in the plane, there exists a Tverberg cycle.
\end{theorem}
\begin{theorem}
\label{theorem:blue-redmatching}
For any set of \( n \) red points and \( n \) blue points in the plane, there exists a Tverberg red-blue matching.
\end{theorem}

Remark that Theorem~\ref{theorem:cycleintheplane} is a refinement of the main result from~\cite{soberon2020tverberg} mentioned earlier: Our approach also works for an even set of points in the plane. Theorem~\ref{theorem:blue-redmatching} is the main result from~\cite{huemer2019matching}, however, we give a new proof which seems to be simpler than the original one. It is worth mentioning that our proofs of these theorems are very similar and this is not a coincidence because of the following connection between them. For any finite set of points in the plane, one can consider the colored multiset containing each point of the set colored in red and also a copy of each point colored in blue. Hence, in this setting, a Tverberg cycle can be viewed as a special red-blue Tverberg matching.

Our main results are higher-dimensional generalizations of the main theorem from~\cite{bereg2019maximum} and~\cite{huemer2019matching}.

\begin{theorem}
\label{theorem:open-Tverberg-R-d}
For any even set of distinct points in \( \mathbb R^d \), there exists an open Tverberg matching.
\end{theorem}

\begin{theorem}
\label{theorem:blue-red-matching-in-r-d}
    For \( n \) red points and \( n \) blue points in \( \mathbb R^d \), there exists a red-blue Tverberg matching.
\end{theorem}

Remark that Theorem~\ref{theorem:blue-redmatching} is a special case of Theorem~\ref{theorem:blue-red-matching-in-r-d}. However, we decided to leave the proofs of both theorems because they are absolutely different. It is worth mentioning that Theorems~\ref{theorem:blue-redmatching} and~\ref{theorem:blue-red-matching-in-r-d} are in a sense tight: Consider the vertices of a square which are colored alternatively in red and blue; for any red-blue matching the intersection of the induced balls is a point.

Our proofs of Theorems~\ref{theorem:open-Tverberg-R-d} and~\ref{theorem:blue-red-matching-in-r-d} are based on the method of infinite descent. Note that both proofs of Tverberg's theorem by Tverberg and Vre{\'{c}}ica~\cite{Tverberg1993} and by Roudneff~\cite{Roudneff2001} give a good illustration of this method. 
Recalling that Tverberg's theorem has a colorful variation (unfortunately, proved only in special cases; see~\cites{Barany1992, Blagojevi2015, Blagojevi2011}), we can also interpret Theorems~\ref{theorem:open-Tverberg-R-d} and~\ref{theorem:blue-red-matching-in-r-d} as a Tverberg-type theorem and its colorful version, respectively.
Remark that the combinatorial and linear-algebraic ingredients of the proofs of Theorems~\ref{theorem:open-Tverberg-R-d} and \ref{theorem:blue-red-matching-in-r-d} are different.

Throughout the paper, we use the standard notation of convex geometry and graph theory; see~the books~\cite{barvinok2002course}~and~\cite{west2001introduction} on convexity and graph theory, respectively.

\medskip

The paper is organized as follows. In Section~\ref{section:prelimitaries for planar results} we introduce the notation and discuss simple observations that we need to prove Theorems~\ref{theorem:cycleintheplane} and \ref{theorem:blue-redmatching} in Sections~\ref{section:prove-theor1} and~\ref{section:blue-red-matching}, respectively. In Subsection~\ref{section:prop-extr-point}, we study properties of the global minimum of a function playing the key role in the proofs of Theorems~\ref{theorem:open-Tverberg-R-d} and~\ref{theorem:blue-red-matching-in-r-d}. In Subsection~\ref{section:prop-obt-graph}, we introduce a concept of obtuse graph and study its properties. Using the lemmas from Subsection~\ref{section:prop-extr-point} and~\ref{section:prop-obt-graph}, we prove Theorem~\ref{theorem:open-Tverberg-R-d} in Section~\ref{section:proof-Theor-3}. Finally, applying the lemma from~Subsection~\ref{section:prop-extr-point}, we give a proof of Theorem~\ref{theorem:blue-red-matching-in-r-d} in Section~\ref{section:proof-perfmat-color}. In Section~\ref{section:discussion}, we discuss the method of infinite descent in the context of Tverberg-type results. Finally, in Section~\ref{section:open_problems}, we propose a number of new open problems.

\subsection*{Acknowledgments.} We thank Pablo Sober{\'o}n for sharing the problem. Besides, we thank J{\'a}nos Pach for drawing our attention to~\cite{dumitrescu2009drawing}. We are also grateful to the members of the Laboratory of Combinatorial and Geometric Structures at MIPT for the stimulating and fruitful discussions. We thanks to the referees for their suggestions that helped us to significantly improve the presentation of the paper.

\section{Preliminaries for planar results}
\label{section:prelimitaries for planar results}

We say that a finite set of points is \textit{in general position} if no two segments spanned by points of the set are orthogonal or parallel.
It is sufficient to prove Theorems~\ref{theorem:cycleintheplane} and~\ref{theorem:blue-redmatching} additionally assuming that the points are in general position. Indeed, for any \(n\)-tuple \( S \) of points in the plane, consider a sequence \( \{S_k\}_{k=1}^{\infty} \) of \(n\)-tuples of points in general position converging to \( S \). Suppose that for each \( S_k \), we can find a proper Tverberg graph, which naturally induces a graph \(G_k\) with the vertex set \(\{1,\dots,n\}\). Hence, there is a graph \( G \) that occurs infinitely many times in the sequence \(\{G_k\}_{k=1}^\infty\). Since there is a compactum containing all disks induces by pairs of points in \( S \) or \( S_k \), the corresponding graph for \( S \) is also a Tverberg graph. Moreover, if each \(G_k\) is a Hamiltonian cycle (or a perfect red-blue matching), so is the obtained graph. Throughout Sections~\ref{section:prelimitaries for planar results},~\ref{section:prove-theor1}, and~\ref{section:blue-red-matching}, we consider only finite point sets in general position in the plane.

For a finite non-empty set \( S \) of points in the plane, a line is called \textit{bisecting} if there are at most \( |S|/2 \) points of \( S \) in each of its open half-planes. 

Choose any unit vector \( v \). Denote by \( v_\alpha \) the unit vector obtained by the counterclockwise rotation of \( v \) by an angle \( \alpha \). Let \( H_\alpha \) be the union of bisecting lines for \( S \) orthogonal to \(v_\alpha \). Therefore, \( H_\alpha\) is a closed plank lying between two parallel bisecting lines, each of those passes through a point of \( S \). Moreover, every line lying between the lines bounding \( H_\alpha \) does not contain a point of \( S \). Denote by \( \ell_\alpha\) the midline of \( H_\alpha\); see Figure \ref{figure:basic}. Since \(H_\alpha=H_{\alpha+\pi}\), we obtain that the lines \(\ell_\alpha \) and \(\ell_{\alpha+\pi}\) coincide. For an odd set \(S\), the plank \(H_\alpha\) coincides with the line \(\ell_\alpha\) passing through a point of \(S\). For an even set \( S \), the plank \( H_\alpha \) coincides with the line \( \ell_\alpha \) if and only if \(\ell_\alpha\) passes through exactly two points of~\(S\). Denote by \(\ell_\alpha^+\) the closed half-plane bounded by \(\ell_\alpha\) such that \( v_\alpha\) is its inner normal vector.

\begin{lemma}
\label{lemma:passing-through-a-point}
    For any point of \( S \), there is a bisecting line \(\ell_\alpha\) passing through it.
\end{lemma}
\begin{proof}
    Consider any point \(p\) in the plane and any bisecting line \(\ell_\beta\). Suppose \(p\) does not lie on \(\ell_\beta\) and without loss of generality, assume that \(p \in \ell_\beta^+\). Since the lines \(\ell_\beta\) and \( \ell_{\beta+\pi} \) coincide, we have \(p\not\in\ell_{\beta+\pi}^+\). As \(\ell_\gamma\) is the midline of \(H_\alpha\), the half-plane \( \ell_\gamma^+ \) continuously depends on \( \gamma \), and thus, there is \( \alpha \) such that the point \(p\) lies on \(\ell_{\alpha}\).
\end{proof}
\begin{corollary}
\label{corollary:even-point-set-bisecting-lines}
    For an even set \(S\) of points, there is a bisecting line \(\ell_\alpha\) passing through two points of \(S\).
\end{corollary}
\begin{proof}
    By Lemma~\ref{lemma:passing-through-a-point}, there is a bisecting line \(\ell_\alpha\) passing through a point of \(S\).
    Since \(\ell_\alpha\) is the midline of \(H_\alpha\), it coincides with \( H_\alpha\), and hence, contains exactly two points of \(S\).
\end{proof}

Denote by \(o_\alpha\) the intersection point of lines \( \ell_\alpha \) and \(\ell_{\alpha+\pi/2}\). These lines determine four \textit{closed} quadrants in the plane. Let us enumerate them in the counterclockwise order by \(1,2,3,4\) starting from \(\ell_\alpha^+\cap\ell_{\alpha+\pi/2}^+\); see Figures ~\ref{figure:theorem-tverberg-cycle} and~\ref{figure:thereom-red-blue-in-the plane}, where we denote the quadrants as \(Q_1,Q_2,Q_3,Q_4\), respectively. 
Remark that a point may belong to two distinct quadrants if and only if it lies on the intersection of their boundaries.

\begin{corollary}
\label{corollary: odd-point-set-bisecting-lines}
    For an odd set \(S\) of points, there are bisecting lines \(\ell_{\alpha}\) and \(\ell_{\alpha+\pi/2}\) such that each of them passes through exactly one of two distinct points of \(S\).
\end{corollary}
\begin{proof}
    Consider two points \(x,y\in S\) at the maximum distance among all pairwise distances between points of \(S\). Clearly, both of them lie on the boundary of the convex hull of \(S\). By Lemma~\ref{lemma:passing-through-a-point}, there is a bisecting line passing through the point \(x\). If it contains two points of \(S\), we may rotate it around \(x\) a little bit in such a way that the resulting bisecting line \(\ell_\alpha\) contains only the point \(x\) of \( S \). We claim that the line \(\ell_{\alpha+\pi/2}\) does not pass through \(x\). Indeed, suppose that \(x=o_\alpha\). Without loss of generality, assume that \(y\) lies in the second quadrant, and thus, there is a point \(z\in S\) distinct from \(x\) lying in fourth quadrant. Since \(x\) is the closest point to \(y\) in the fourth quadrant, we have \(\|z-y\|>\|x-y\|\), a contradiction with maximality of the distance between \(x\) and \(y\), see Figure \ref{figure:corollary_eight}. Hence, the point \(x\) does not lie on \(\ell_{\alpha+\pi/2}\), and so, this line passes through another point of \(S\). If it contains two points of \(S\), then we may rotate \(\ell_\alpha\) and \(\ell_{\alpha+\pi/2}\) in such a way that each of them contains exactly one point of \(S\).
\end{proof}
\begin{figure}
    \centering
    \begin{subfigure}[t]{.49\textwidth}
    \centering
        \includegraphics[width=.7\columnwidth]{figure_two_a.mps}
        \caption{Plank~$H_{\alpha}$~and~its~midline~\(\ell_\alpha\).}
        \label{figure:basic}
    \end{subfigure}
    \begin{subfigure}[t]{.49\textwidth}
        \centering
        \includegraphics[width=.7\columnwidth]{figure_two_b.mps}
        \caption{ Proof~of~Corollary~\ref{corollary: odd-point-set-bisecting-lines}.}
        \label{figure:corollary_eight}
    \end{subfigure}
    \caption{}
\end{figure}

Next, we state a trivial observation playing the crucial role in proving Theorems~\ref{theorem:cycleintheplane} and~\ref{theorem:blue-redmatching}.
\begin{keyobservation}
\label{key-observation}
The point \( o_\alpha \) lies in any disk with a diameter whose endpoints lie in two opposite quadrants.
\end{keyobservation}

\section{Proof of Theorem~\ref{theorem:cycleintheplane}}
\label{section:prove-theor1}

Let \( S \) be a set of \( n \) points in the plane. Consider two bisecting lines \(\ell_\alpha\) and \(\ell_{\alpha+\pi/2}\) orthogonal to each other. Later, we specify the choice of these lines. Denote by \(s_i\) the number of points in the \(i\)-th closed quadrant. 

There are several possible cases.

\begin{figure}[h]
  \centering
\begin{tikzpicture}[scale=7, line cap=round,line join=round,>=triangle 45,x=1.0cm,y=2.0cm]
\begin{scope}
\draw [line width=2pt] (0.22,-2.37)-- (0.89,-2.37);
\draw [line width=2pt] (0.54,-2.21)-- (0.54,-2.54);
\draw (0.46,-2.33) node[anchor=north west] {$o$};
\draw (0.84,-2.37) node[anchor=north west] {$\ell_{\alpha}$};
\draw (0.8,-2.21) node[anchor=north west] {$Q_1$};
\draw (0.8,-2.52) node[anchor=north west] {$Q_4$};
\draw (0.22,-2.52) node[anchor=north west] {$Q_3$};
\draw (0.22,-2.21) node[anchor=north west] {$Q_2$};
\draw [line width=2pt] (0.97,-2.37)-- (1.64,-2.37);
\draw [line width=2pt] (1.3,-2.21)-- (1.3,-2.54);
\draw (1.22,-2.33) node[anchor=north west] {$o$};
\draw (1.59,-2.37) node[anchor=north west] {$\ell_{\alpha}$};
\draw (1.57,-2.21) node[anchor=north west] {$Q_1$};
\draw (1.57,-2.52) node[anchor=north west] {$Q_4$};
\draw (0.98,-2.52) node[anchor=north west] {$Q_3$};
\draw (0.98,-2.21) node[anchor=north west] {$Q_2$};
\draw [line width=2pt] (1.74,-2.37)-- (2.41,-2.37);
\draw [line width=2pt] (2.06,-2.21)-- (2.06,-2.53);
\draw (1.99,-2.33) node[anchor=north west] {$o_\alpha$};
\draw (2.36,-2.37) node[anchor=north west] {$\ell_{\alpha}$};
\draw (2.33,-2.21) node[anchor=north west] {$Q_1$};
\draw (2.33,-2.52) node[anchor=north west] {$Q_4$};
\draw (1.74,-2.52) node[anchor=north west] {$Q_3$};
\draw (1.74,-2.21) node[anchor=north west] {$Q_2$};
\draw (0.37,-2.38) node[anchor=north west] {$x$};
\draw (0.66,-2.38) node[anchor=north west] {$y$};
\draw (1.41,-2.37)-- (1.05,-2.27);
\draw (1.05,-2.27)-- (1.47,-2.48);
\draw (1.47,-2.48)-- (1.21,-2.25);
\draw (1.21,-2.25)-- (1.51,-2.37);
\draw (1.41,-2.37)-- (1.04,-2.47);
\draw (1.04,-2.47)-- (1.5,-2.27);
\draw (1.5,-2.27)-- (1.22,-2.51);
\draw (1.51,-2.37)-- (1.22,-2.51);
\draw (0.41,-2.37)-- (0.7,-2.25);
\draw (0.7,-2.25)-- (0.37,-2.49);
\draw (0.37,-2.49)-- (0.78,-2.28);
\draw (0.78,-2.28)-- (0.29,-2.45);
\draw (0.29,-2.45)-- (0.70,-2.37);
\draw (0.70,-2.37)-- (0.37,-2.24);
\draw (0.77,-2.5)-- (0.37,-2.24);
\draw (0.77,-2.5)-- (0.41,-2.37);
\draw (1.91,-2.37)-- (2.22,-2.28);
\draw (2.22,-2.28)-- (1.83,-2.48);
\draw (1.83,-2.48)-- (2.31,-2.32);
\draw (2.31,-2.32)-- (1.91,-2.51);
\draw (1.91,-2.51)-- (2.06,-2.27);
\draw (2.06,-2.27)-- (2.32,-2.47);
\draw (2.32,-2.47)-- (1.92,-2.29);
\draw (1.92,-2.29)-- (2.2,-2.51);
\draw (2.2,-2.51)-- (1.91,-2.37);
\draw (1.36,-2.38) node[anchor=north west] {$x$};
\draw (1.47,-2.38) node[anchor=north west] {$y$};
\draw (1.86,-2.38) node[anchor=north west] {$x$};
\draw (1.98,-2.25) node[anchor=north west] {$y$};
\begin{scriptsize}
\fill [color=black] (0.41,-2.37) circle (0.45pt);
\fill [color=black] (0.70,-2.37) circle (0.45pt);
\fill [color=black] (0.37,-2.24) circle (0.25pt);
\fill [color=black] (0.7,-2.25) circle (0.25pt);
\fill [color=black] (0.37,-2.49) circle (0.25pt);
\fill [color=black] (0.77,-2.5) circle (0.25pt);
\fill [color=black] (1.41,-2.37) circle (0.45pt);
\fill [color=black] (1.21,-2.25) circle (0.25pt);
\fill [color=black] (1.51,-2.37) circle (0.45pt);
\fill [color=black] (1.5,-2.27) circle (0.25pt);
\fill [color=black] (1.04,-2.47) circle (0.25pt);
\fill [color=black] (1.47,-2.48) circle (0.25pt);
\fill [color=black] (1.05,-2.27) circle (0.25pt);
\fill [color=black] (1.22,-2.51) circle (0.25pt);
\fill [color=black] (0.29,-2.45) circle (0.25pt);
\fill [color=black] (0.78,-2.28) circle (0.25pt);
\fill [color=black] (1.91,-2.37) circle (0.45pt);
\fill [color=black] (2.06,-2.27) circle (0.45pt);
\fill [color=black] (2.22,-2.28) circle (0.25pt);
\fill [color=black] (2.31,-2.32) circle (0.25pt);
\fill [color=black] (1.83,-2.48) circle (0.25pt);
\fill [color=black] (1.91,-2.51) circle (0.25pt);
\fill [color=black] (1.92,-2.29) circle (0.25pt);
\fill [color=black] (2.2,-2.51) circle (0.25pt);
\fill [color=black] (2.32,-2.47) circle (0.25pt);
\end{scriptsize}
\end{scope}
\end{tikzpicture}
\caption{\textit{Case 1.1} (\(n=8\)),  \textit{Case 1.2} (\(n=8\)),  \textit{Case 2} (\(n=9\)).}\label{figure:theorem-tverberg-cycle}
\end{figure}
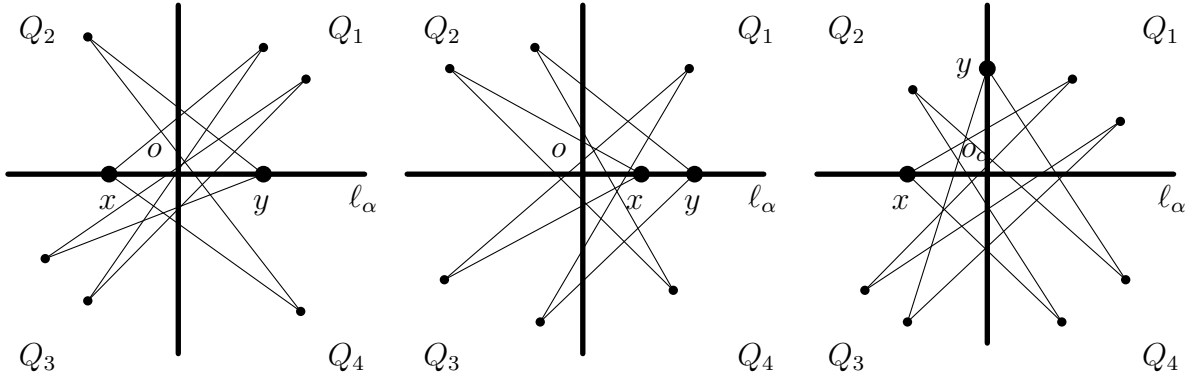

\textit{Case 1.} Let \( n \) be even. By Corollary~\ref{corollary:even-point-set-bisecting-lines}, we may choose a bisecting line \( \ell_\alpha \) passing through exactly two points of \( S \). Denote these points by \( x \) and \( y \). Since no two segments spanned by points of \( S \) are orthogonal, the bisecting line \( \ell_{\alpha+\pi/2}\) contains no points of \( S \). There are two possible cases depending on the arrangement of the points \( x,y,\) and \( o_\alpha \) on the line \( \ell_\alpha \).

\textit{Case 1.1.} The point \( o_\alpha \) lies between the points \( x \) and \( y \) on the line \( \ell_\alpha \). Without loss of generality, assume that \( x \) belongs to the second and third quadrants and the point \( y \) belongs to the first and fourth quadrants. Since \( \ell_\alpha \) and \( \ell_{\alpha + \pi/2}\) are bisecting lines, we obtain that \( s_1=s_3 \) and \( s_2=s_4\). Therefore, we easily construct a Hamiltonian cycle with edges satisfying Observation~\ref{key-observation}: We run alternatively between the points of the third and first quadrants starting from \( x \) and finishing at \(y\), and then we continue to run alternatively between the points of the fourth and second quadrants starting from \( y \) and finishing at \(x\).

\textit{Case 1.2.} The points \( x \) and \( y \) lie on the same side with respect to \( o_\alpha \) on the line \( \ell_\alpha \). Without loss of generality assume that \(x\) and \(y\) belong to the first and fourth quadrants. Using the fact that \( \ell_\alpha \) and \( \ell_{\alpha + \pi/2} \) are bisecting lines, we obtain that \(s_1=s_3+1\) and \(s_4=s_2+1\). Therefore, we easily construct a Hamiltonian cycle with edges satisfying Observation~\ref{key-observation}: We run alternatively between the points of the first and third quadrants starting from \(x\) and finishing at \(y\), and then we continue to run alternatively between the points of the fourth and second quadrants starting from \( y \) and finishing at \( x \).

\textit{Case 2.} Let \( n \) be odd. By Corollary~\ref{corollary: odd-point-set-bisecting-lines}, we may choose bisecting lines \( \ell_{\alpha} \) and \( \ell_{\alpha+\pi/2} \) such that each of them passes through exactly one of two distinct points of \( S \). Denote by \( x\) the point of \(S\) lying on \(\ell_\alpha\) and by \(y\) the point of \(S \) lying on \( \ell_{\alpha + \pi/2} \).

Without loss of generality, assume that the point \( x \) belongs to the second and third quadrants and the point \(y \) belongs to the first and second quadrants. Since \( \ell_\alpha \) and \( \ell_{\alpha+\pi/2}\) are bisecting lines, we obtain that \( s_1=s_3\) and \(s_2=s_4+1\). Therefore, we easily construct a Hamiltonian cycle with edges satisfying Observation~\ref{key-observation}: We run alternatively between the points of the third and first quadrants starting from \(x\) and finishing at \(y\), and then we continue to run alternatively between the points of the second and fourth
quadrants starting from \( y \) and finishing at \(x\). \hfill \(\square\)

\begin{remark}
Recently the following conjecture of Fekete and Woeginger~\cite{fekete1997angle} resembling Theorem~\ref{theorem:cycleintheplane} was confirmed by Biniaz~\cite{biniaz2021acute}: 
For any sufficiently large even number \( n \), every set \(S\) of \( n \) points in the plane can be connected by a Hamiltonian cycle consisting of straight-line edges such that the angle between any two consecutive edges is at most \( \pi/2 \). 
As the first step of the proof he applied the idea from~\cite{dumitrescu2009drawing} similar to our approach. Namely, he considered two orthogonal bisecting lines for \(S\) that partition the plane into four quadrants containing almost the same number of points in \(S\) and used the following observation (compare with Observation~\ref{key-observation}): For any two distinct points \(x, z\) lying in one quadrant and a point \(y\) from the opposite quadrant, the angle \(\angle xyz\) is acute. Also, note that the Tverberg cycle obtained in the proof of Theorem~\ref{theorem:cycleintheplane} contains a lot of acute angles, however, not necessarily all its angles are acute.
\end{remark}

\section{Proof of Theorem~\ref{theorem:blue-redmatching}}
\label{section:blue-red-matching}

Denote by \( S \) the set of \( n \) red points and \( n \) blue points.

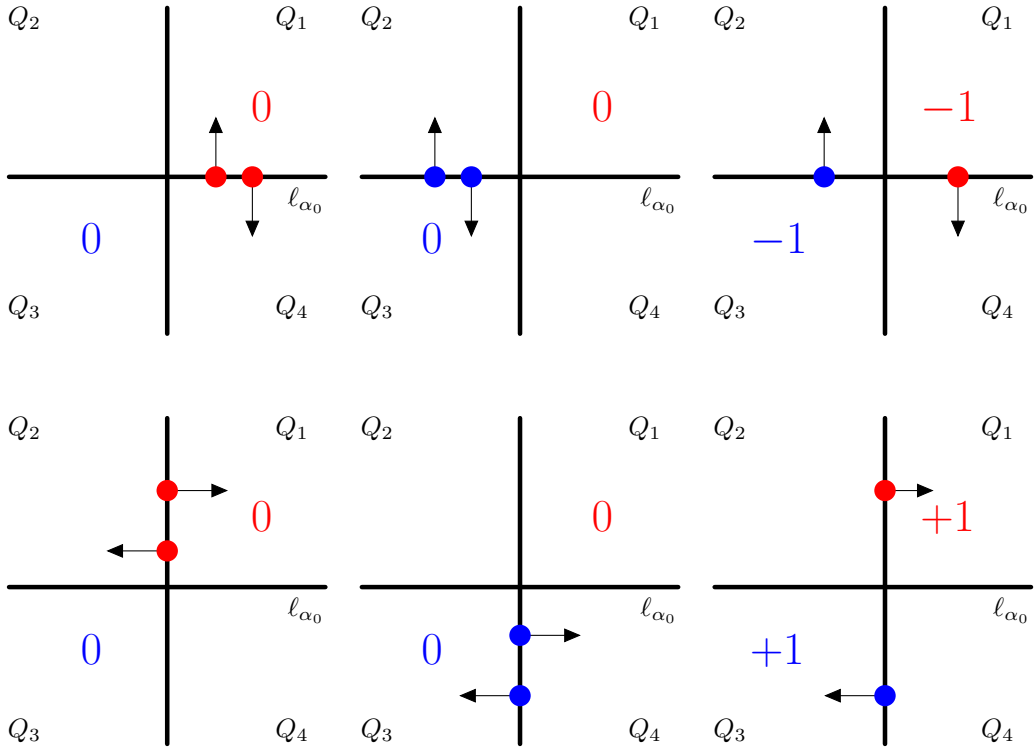
\begin{figure}[h]
  \centering
\begin{tikzpicture}[scale=16, line cap=round,line join=round,>=triangle 45,x=1.0cm,y=1.0cm]
\draw [line width=1.7pt] (0.39,-0.46)-- (0.65,-0.46);
\draw [line width=1.7pt] (0.52,-0.59)-- (0.52,-0.32);
\draw [->] (0.56,-0.46) -- (0.56,-0.41);
\draw [->] (0.59,-0.46) -- (0.59,-0.51);
\draw [line width=1.7pt] (0.39,-0.8)-- (0.65,-0.8);
\draw [line width=1.7pt] (0.52,-0.93)-- (0.52,-0.66);
\draw [->] (0.52,-0.72) -- (0.57,-0.72);
\draw [->] (0.52,-0.77) -- (0.47,-0.77);
\draw [line width=1.7pt] (0.68,-0.46)-- (0.94,-0.46);
\draw [line width=1.7pt] (0.81,-0.59)-- (0.81,-0.32);
\draw [->] (0.74,-0.46) -- (0.74,-0.41);
\draw [->] (0.77,-0.46) -- (0.77,-0.51);
\draw [line width=1.7pt] (0.68,-0.8)-- (0.94,-0.8);
\draw [line width=1.7pt] (0.81,-0.93)-- (0.81,-0.66);
\draw [->] (0.81,-0.84) -- (0.86,-0.84);
\draw [->] (0.81,-0.89) -- (0.76,-0.89);
\draw [line width=1.7pt] (0.97,-0.46)-- (1.23,-0.46);
\draw [line width=1.7pt] (1.11,-0.59)-- (1.11,-0.32);

\draw [->] (1.06,-0.46) -- (1.06,-0.41);
\draw [->] (1.17,-0.46) -- (1.17,-0.51);
\draw [line width=1.7pt] (0.97,-0.8)-- (1.23,-0.8);
\draw [line width=1.7pt] (1.11,-0.93)-- (1.11,-0.66);

\tikzstyle{every node}=[font=\fontsize{10}{30}\selectfont]

\draw (0.60,-0.31) node[anchor=north west] {$Q_1$};
\draw (0.60,-0.55) node[anchor=north west] {$Q_4$};
\draw (0.38,-0.55) node[anchor=north west] {$Q_3$};
\draw (0.38,-0.31) node[anchor=north west] {$Q_2$};
\draw (0.60,-0.65) node[anchor=north west] {$Q_1$};
\draw (0.60,-0.90) node[anchor=north west] {$Q_4$};
\draw (0.38,-0.90) node[anchor=north west] {$Q_3$};
\draw (0.38,-0.65) node[anchor=north west] {$Q_2$};
\draw (0.89,-0.31) node[anchor=north west] {$Q_1$};
\draw (0.89,-0.55) node[anchor=north west] {$Q_4$};
\draw (0.67,-0.55) node[anchor=north west] {$Q_3$};
\draw (0.67,-0.31) node[anchor=north west] {$Q_2$};
\draw (0.89,-0.65) node[anchor=north west] {$Q_1$};
\draw (0.89,-0.90) node[anchor=north west] {$Q_4$};
\draw (0.67,-0.90) node[anchor=north west] {$Q_3$};
\draw (0.67,-0.65) node[anchor=north west] {$Q_2$};
\draw (1.18,-0.31) node[anchor=north west] {$Q_1$};
\draw (1.18,-0.55) node[anchor=north west] {$Q_4$};
\draw (0.96,-0.55) node[anchor=north west] {$Q_3$};
\draw (0.96,-0.31) node[anchor=north west] {$Q_2$};
\draw (1.18,-0.65) node[anchor=north west] {$Q_1$};
\draw (1.18,-0.90) node[anchor=north west] {$Q_4$};
\draw (0.96,-0.90) node[anchor=north west] {$Q_3$};
\draw (0.96,-0.65) node[anchor=north west] {$Q_2$};

\draw [->] (1.11,-0.72) -- (1.15,-0.72);
\draw [->] (1.11,-0.89) -- (1.06,-0.89);

\begin{scriptsize}
\fill [color=red] (0.56,-0.46) circle (0.25pt);
\fill [color=red] (0.59,-0.46) circle (0.25pt);
\fill [color=red] (0.52,-0.72) circle (0.25pt);
\fill [color=red] (0.52,-0.77) circle (0.25pt);
\fill [color=blue] (0.74,-0.46) circle (0.25pt);
\fill [color=blue] (0.77,-0.46) circle (0.25pt);
\fill [color=blue] (0.81,-0.84) circle (0.25pt);
\fill [color=blue] (0.81,-0.89) circle (0.25pt);
\fill [color=blue] (1.06,-0.46) circle (0.25pt);
\fill [color=red] (1.17,-0.46) circle (0.25pt);
\fill [color=red] (1.11,-0.72) circle (0.25pt);
\fill [color=blue] (1.11,-0.89) circle (0.25pt);
\end{scriptsize}

\tikzstyle{every node}=[font=\fontsize{16}{30}\selectfont]

\draw [color=red](0.58,-0.38) node[anchor=north west] {$0$};
\draw [color=blue](0.44,-0.49) node[anchor=north west] {$0$};
\draw [color=red](0.86,-0.38) node[anchor=north west] {$0$};
\draw [color=blue](0.72,-0.49) node[anchor=north west] {$0$};
\draw [color=red](1.13,-0.38) node[anchor=north west] {$-1$};
\draw [color=blue](0.99,-0.49) node[anchor=north west] {$-1$};
\draw [color=red](0.58,-0.72) node[anchor=north west] {$0$};
\draw [color=blue](0.44,-0.83) node[anchor=north west] {$0$};
\draw [color=red](0.86,-0.72) node[anchor=north west] {$0$};
\draw [color=blue](0.72,-0.83) node[anchor=north west] {$0$};
\draw [color=red](1.13,-0.72) node[anchor=north west] {$+1$};
\draw [color=blue](0.99,-0.83) node[anchor=north west] {$+1$};

\tikzstyle{every node}=[font=\fontsize{10}{30}\selectfont]

\draw (0.61,-0.46) node[anchor=north west] {$ \ell_{\alpha_0}$};
\draw (0.90,-0.46) node[anchor=north west] {$ \ell_{\alpha_0}$};
\draw (1.19,-0.46) node[anchor=north west] {$ \ell_{\alpha_0}$};
\draw (0.61,-0.8) node[anchor=north west] {$ \ell_{\alpha_0}$};
\draw (0.90,-0.8) node[anchor=north west] {$ \ell_{\alpha_0}$};
\draw (1.19,-0.8) node[anchor=north west] {$ \ell_{\alpha_0}$};

\end{tikzpicture}
  \caption{Illustrations showing that \( |F(\alpha_0+\varepsilon)-F(\alpha_0-\varepsilon)|\leq 1\).}\label{figure:thereom-red-blue-in-the plane}
\end{figure}

Using the notation of Section~\ref{section:prelimitaries for planar results}, consider a set \(\Omega\) of \(\alpha\in \mathbb R / 2\pi\mathbb Z\) such that none of the bisecting lines \( \ell_{\alpha}\) and \( \ell_{\alpha+\pi/2}\) for \( S \) passes through points of \(S\). Remark that the complementary set \( \overline{\Omega} = (\mathbb R/ 2\pi \mathbb Z) \setminus \Omega \) is finite. Moreover, for every \( \alpha \in \overline {\Omega}\), only one of two lines \( \ell_\alpha \) or \( \ell_{\alpha+\pi/2} \) passes through a point of \( S \) (in fact, it passes through exactly two points of \( S \)).

For \( \alpha \in \Omega \), denote the number of red and blue points in the \( i \)-th quadrant by \( r_i(\alpha) \) and \( b_i(\alpha)\), respectively. We claim that
\begin{equation}
\label{equation:red-blue-equality}
    r_1(\alpha) - b_3(\alpha) =  r_3(\alpha) - b_1(\alpha) = b_2(\alpha) - r_2(\alpha) = b_4(\alpha) - r_2(\alpha).
\end{equation}

Indeed, since the lines \( \ell_\alpha \) and \( \ell_{\alpha+\pi/2} \) are bisecting, we have \( r_1(\alpha) + b_1(\alpha) = r_3(\alpha) + b_3(\alpha)\) and \(r_2(\alpha)+b_2(\alpha)=r_4(\alpha)+b_4(\alpha)\). Consequently, 
\begin{equation}
\label{equation:red-blue-semi}
r_1(\alpha) - b_3(\alpha) = r_3(\alpha) - b_1(\alpha)\ \ \text{ and }\ \ b_2(\alpha) - r_4(\alpha) = b_4(\alpha) - r_2(\alpha) .
\end{equation}
Next, from the equality of the numbers of red and blue points we have
\[
    r_1(\alpha) + r_2(\alpha) + r_3(\alpha) + r_4(\alpha) = |S|/2= n = b_1(\alpha) + b_2(\alpha) + b_3(\alpha) + b_4(\alpha),
\]
and thus,
\[ 
    \big(r_1(\alpha)-b_3(\alpha) \big) + \big(r_3(\alpha) - b_1 (\alpha)\big) =
    \big(b_2(\alpha)-r_4(\alpha) \big) + \big(b_4(\alpha)-r_2(\alpha)\big).
\]
Combining this equality with (\ref{equation:red-blue-semi}), we easily obtain~(\ref{equation:red-blue-equality}).

According to Observation~\ref{key-observation}, it is enough to prove that there is \(\alpha\in \Omega\) such that the number of red points in any of quadrants equals to the number of blue points in its opposite quadrant. By~(\ref{equation:red-blue-equality}), it remains to show that \( r_1(\alpha)=b_3(\alpha) \) for some \( \alpha \in \Omega \). 

Finally, consider the function \( F:\Omega \to \mathbb Z\) defined by \( F ( \alpha ) = r_1( \alpha )-b_3( \alpha ) \). Since \( \ell_\alpha \) is a midline of \(H_\alpha\), it changes continuously for \( \alpha \in \mathbb R / 2\pi \mathbb Z \), and thus, we get that \( F \) is a piecewise constant function. Moreover, it is constant between any two consecutive points of  \( \overline{\Omega} \). We claim that \( F \) changes its value at points of \( \overline{\Omega} \) by 1, 0, or \(-1\). Indeed, consider \( \alpha_0 \in \overline{\Omega} \). Let us change \(\alpha\) continuously from \( \alpha_0 - \varepsilon \) to \(\alpha_0 + \varepsilon\), where \(\varepsilon>0\) is chosen in such a way that \( \alpha_0 \) is the only point of \( \overline{\Omega}\) lying between \( \alpha_0-\varepsilon \) and \( \alpha_0+\varepsilon\). We know that one of the lines \(\ell_{\alpha_0}\) or \(\ell_{\alpha_0+\pi/2}\) contains exactly two points of \(S\) and the other does not pass through any point of \( S \). Therefore, \(|F(\alpha_0+\varepsilon)-F(\alpha_0-\varepsilon) |\leq 2\). Notice that \(|F(\alpha_0+\varepsilon)-F(\alpha_0-\varepsilon) |\) can be equal to 2 only if one of the following properties holds for \( \alpha=\alpha_0 \): (\textit{i}) the boundary of the first quadrant contains two red points; (\textit{ii}) the boundary of the third quadrant contains two blue points; (\textit{iii}) the boundary of the first quadrant contains one red point and the boundary of the third quadrant contains one blue point. There are 6 possible arrangements of these two points of \( S \) and the point \( o_{\alpha_0} \) satisfying one of these properties. 
We leave to the reader as a simple exercise the exhaustion verification of the fact that \(F\) changes by \(1, 0 \) or \(-1\) at the point \( \alpha_0 \); see Figure~\ref{figure:thereom-red-blue-in-the plane}.

Without loss of generality, assume that \( 0 \in \Omega \). Since \(  \ell_{\alpha} = \ell_{\alpha+\pi} \) and \( \ell_{\alpha+\pi/2} = \ell_{ \alpha+3\pi/2 } \), we have 
\begin{align*}
    F(0)&=r_1(0)- b_3(0), & 
    F(\pi / 2)&= r_1(\pi / 2)&-&&b_3(\pi / 2)=& r_2(0) - b_4(0),
\\
    F(\pi)= r_1(\pi)-b_3(\pi) &= r_3(0) - b_1(0), &  
    F(3\pi / 2)&= r_1(3\pi / 2)&-&&b_3(3\pi / 2)=& r_4(0) - b_2(0).
\end{align*}
From the equality of the numbers of red and blue points, we get that the equality \( F(0)+F(\pi/2)+F(\pi)+F(3\pi/2)=0 \) holds, and thus, the function \( F \) takes non-negative and non-positive values. Since the function \( F \) takes only integer values and changes its value at a finite number of points by at most 1, there is a value of \( \alpha \in \Omega \), where \( F \) vanishes. Using the corresponding pairs of orthogonal lines \( \ell_{\alpha}\) and \( \ell_{\alpha+\pi/2} \) and Observation~\ref{key-observation}, we construct the desired matching.
\hfill \(\square\)

\section{Preliminaries for high-dimensional results}

Throughout the remaining sections, we denote the origin of \( \mathbb R^d \) by \( o \).

\subsection{Extreme point of the maximum of dot products.}
\label{section:prop-extr-point}
For a finite set \( \mathcal M\) of pairs of points in \(\mathbb R^d\), consider the function \(H_{\mathcal M}: \mathbb R^d \to \mathbb R\) defined by
\[
H_{\mathcal M}(x) = \max\big\{ \langle a-x, b-x \rangle: ab \in \mathcal M\big\}.
\]

Note that if this function attains a non-positive (negative) value at some point \(x_0\in \mathbb R^d\), then this point \(x_0\) is a common point of all (open) balls with diameters induced by the pairs in \( \mathcal M \), and hence, \(\mathcal M\) is a (an open) Tverberg matching for the set of all points in pairs of~\(\mathcal M\).
\begin{lemma}
Let \(\mathcal M\) be a finite set of pairs of points in \(\mathbb R^d\). Then the function \( H_{\mathcal M}(x) \) attains its strict global minimum at a unique point \(x_{\mathcal M}\). Moreover, if \( \mathcal M_0 \) is the subset of \( \mathcal M \) consisting of all pairs \(ab\) such that \( \langle a-x_{\mathcal M}, b-x_{\mathcal M} \rangle = H_{\mathcal M} (x_{\mathcal M}) \) and \( K \) is the set of the midpoints of pairs in \( \mathcal M_0 \), then \( x_{\mathcal M} \in \conv K\). 
\label{lemma:point_in_convex_hull}
\end{lemma}
\begin{proof}
Since \( \langle a-x, b-x \rangle =\left\|\frac{a+b}{2}-x\right\|^2-\left\|\frac{a-b}{2}\right\|^2 \), the function $H_{\mathcal M}$ is strictly convex and bounded from below, and thus, it attains its strict global minimum at a unique point~$x_{\mathcal M}$. Without loss of generality we may assume that \( x_{\mathcal M} \) coincides with the origin \( o \). 
Suppose to the contrary that \( o\not \in \conv K \). By the separation theorem~\cite{barvinok2002course}, there exists a non-zero vector \( t \) such that 
\(\langle c, t \rangle > 0 \) for any \(c\in \conv K\). Hence, for sufficiently small \( \varepsilon>0\), the point \(\varepsilon t \) is closer to each point of \( K \) than the origin, and thus, for any pair \( ab\in \mathcal M_0 \), the dot product \( \langle a-\varepsilon t, b- \varepsilon t \rangle \) is strictly less than \( H_{\mathcal M}(o) \). Also, for sufficiently small \( \varepsilon>0 \) and any \(ab\in \mathcal M\setminus \mathcal M_0\), the value of \( \langle a - \varepsilon t, b - \varepsilon t \rangle \) does not exceed \( H_{ \mathcal M }(o) \). Therefore, there is a point \(\varepsilon t \) such that \(  H_{\mathcal M}(\varepsilon t)< H_{\mathcal M}(o)\), which contradicts the fact that \( H_{\mathcal M} \) attains its global minimum at \( o \).  
\end{proof}

For the sake of brevity, we call the matching \(\mathcal M_0\) from Lemma~\ref{lemma:point_in_convex_hull} \textit{balancing} for \(\mathcal M\).

\subsection{Properties of an obtuse graph}
\label{section:prop-obt-graph}

We call a finite set \( V \subset \mathbb R^d \) \textit{dependent}
if there are positive coefficients \(\lambda_v\) for \(v\in V\) such that
\[
    \sum_{v \in V} \lambda_v v = o.
\]

We say that a graph \( G \) is \textit{obtuse} if its vertex set is a dependent set
and its vertices \(a,b\) are adjacent if and only if \( \langle a, b \rangle <0 \).
Let us show some properties of an obtuse graph.

\begin{lemma}
\label{lemma:isolated}
    A vertex of an obtuse graph \( G \) is isolated if and only if it coincides with the origin \(o\).
\end{lemma}
\begin{proof}
If the origin is a vertex of \(G\), then it is isolated because of \(\langle v,o\rangle = 0 \) for any \(v\in V(G)\). Hence, it remains to show that any non-zero vertex \(v\) is not isolated. Indeed, there are positive \( \lambda_u \) for \( u \in V(G) \) such that 
\[
    \sum_{u\in V(G)} \lambda_u u=o.
\]
Considering the dot product of this vector and \( v \), we obtain
\[ 
    \lambda_v \langle v, v \rangle + \sum_{u\in V(G)\setminus\{v\}} \lambda_u \langle u, v \rangle =0.
\]
Since \( \lambda_w > 0 \) for \( w \in V(G) \), there is a vertex \( u \in V(G)\setminus \{v\}\) such that \( \langle u, v \rangle <0\), and thus, the vertex \(v\) is not isolated.
\end{proof}

\begin{lemma}
    \label{lemma:orthogonality}
    Any two vertices from different connected components of an obtuse graph $G$ are orthogonal.
\end{lemma}
\begin{proof}
Denote by \(U\subset V(G)\) the set of vertices of some connected component. Put \(W=V(G)\setminus U\). We show that any vector from \(U\) is orthogonal to any vector of \(W\). Since \(V(G)\) is a dependent set, there are positive \(\lambda_v\) for all \(v\in V(G)\) such that
\[
    \sum_{u\in U} \lambda_u u= -\sum_{w\in W} \lambda_w w.
\]
Therefore, we have
\[
    \Big\langle \sum_{u\in U}\lambda_uu, \sum_{w\in W}\lambda_w w\Big\rangle\leq 0.
\]
Since \(\lambda_v>0\) for all \(v\in V(G)\) and there are no edges between \(U\) and \(W\), that is, \(\langle u, w\rangle \geq 0 \) for all \(u\in U\) and \(w\in W\), we obtain \(\langle u, w\rangle =0 \) for all \(u\in U\) and \( w\in W \).
\end{proof}

\section{Proof of Theorem~\ref{theorem:open-Tverberg-R-d}}
\label{section:proof-Theor-3}

Let \( S \) be an even set of distinct points in \( \mathbb R^d \). For any perfect matching \( \mathcal{M} \) on \( S \), we consider the function \( H_{\mathcal M} \) from Subsection~\ref{section:prop-extr-point}. Recall that
\[
H_{\mathcal M}(x) = \max\big \{\langle a-x,b-x\rangle: ab\in \mathcal M \big\}.
\]
If \( H_{\mathcal M}(x) < 0\)  for some matching \( \mathcal M\) and some \( x\in \mathbb R^d \), then the point \(x\) is a common interior point of the balls \(\Ball (ab)\), where \(ab\in \mathcal M\), and so, \( \mathcal M \) is an open Tverberg matching. Hence, suppose to the contrary that for any perfect matching \( \mathcal M \) and any \( x\in \mathbb R^d\), we have \( H_{ \mathcal M }(x) \geq 0 \). Consider the function \(P \) depending on a matching \( \mathcal M \) defined by
\[
    P(\mathcal M):=\inf_{x\in\mathbb R^d} H_{\mathcal M}(x)=\min_{x\in \mathbb R^d} H_{\mathcal M}(x).
\]

Among all perfect matchings \( \mathcal M \) for \( S \), choose a matching \( \mathcal M_1 \) such that 
\[
P(\mathcal M_1)=\min_{\mathcal M} P(\mathcal M) \geq 0. 
\]
For the sake of brevity, put \( m:=P(\mathcal M_1) \). Additionally assume that among all perfect matchings \( \mathcal M \) with \(P(\mathcal M)= m \), the balancing matching \( \mathcal M_2 \) for \(\mathcal M_1\) has the minimum size. According to Lemma~\ref{lemma:point_in_convex_hull}, for \( \mathcal M_1\), there is a unique point \(x_{\mathcal M}\) such that \( H_{\mathcal M_1}(x_{\mathcal M_1}) = P(\mathcal M_1)=m \), and without loss of generality, we may assume that \( x_{\mathcal M_1} \) coincides with the origin~\( o \).
According to Lemma~\ref{lemma:point_in_convex_hull}, the midpoints  of some submatching \( \mathcal M_3 \subseteq \mathcal M_2 \) form a dependent set. Let \( S_1 \subseteq S \) be the union of the points of the pairs from \( \mathcal M_3 \). Hence the set \( S_1 \) is dependent as well.

Consider the obtuse graph \( G \) with the vertex set \( S_1 \). Let \( G_1,\dots, G_r \) be its connected components. Since we assume that all points of \( S \) are distinct, Lemma~\ref{lemma:isolated} yields that there is at most one isolated vertex. If there is a vertex of \(G\) coinciding with \( o \), then we set \( v_1 = o \), otherwise let \( v_1 \) be any vertex of \( G \). Without loss of generality, assume that \( v_1 \) is a vertex of \( G_1 \). By Lemma~\ref{lemma:isolated}, the connected components \( G_i\) for \( i\in \{2,\dots, r\} \) contains at least two vertices, and thus, for each \(i\in \{2,\dots, r\}\), we may choose a vertex \(v_i\)~from \(G_i\) such that \( v_1v_i \not \in \mathcal M_3 \). By Lemma~\ref{lemma:orthogonality}, we have \( \langle v_1, v_i \rangle =0 \) for \(i\in \{2,\dots, r\} \). 
Put
\[
    E_{\pi/2}=\big\{ v_1v_i: i\in \{2,\dots, r\} \big\}.
\]

Consider the graph \( G_0 \) on the vertex set \(S_1\) and with the edge set
\[
E(G) \cup E_{\pi/2} \cup \mathcal M_3.
\]
Remark that the sets \( E(G)\cup E_{\pi/2} \) and \( \mathcal M_3 \) do not intersect. Let us color the edges in \( E(G)\cup E_{\pi/2} \) in red and the edges in \( \mathcal M_3 \) in blue.

Recall that an edge \( e \) of a connected graph \( H \) is called a \textit{cut edge}, if \( H - e \) is a disconnected graph. The key combinatorial ingredient of our proof is the following result of Grossman and H\"aggkvist; see~Corollary~1 in \cite{Grossman1983}.

\begin{lemma}
\label{lemma:lemma-main}
Let $M$ be a perfect matching in a graph $H$. If no edge of $M$ is a cut edge of $H$, then $H$ has a cycle whose edges are taken alternately from
$M$ and $H - M$.
\end{lemma}

In the graph \( G_0 \), the blue edges form a perfect matching and no blue edge is a cut edge because \(v_1\) is connected to all components of \( G \). Hence, the graph \( G_0 \) satisfies the conditions of the Lemma~\ref{lemma:lemma-main}, and thus, there is an alternating cycle 
with blue and red edges. Denote by \( \mathcal R \) and \( \mathcal B \) the sets of red and blue edges of this cycle, respectively.

Consider the following perfect matching
    \[
    \mathcal M_4=\left(\mathcal M_1\setminus \mathcal B\right)\cup \mathcal R.
    \]
    Since \(\langle a,b \rangle \le 0\) for any \( ab \in \mathcal R \), we have $H_{\mathcal M_4}(o) \le H_{\mathcal M_1}(o)=m$. Therefore, \( H_{\mathcal M_4}(o)=m\) and the origin \( o \) coincides with the point \( x_{\mathcal M_4}\).
    Recall that \( \mathcal B \subseteq \mathcal M_3\subseteq \mathcal M_2\) and the set \( \mathcal R \) contains at most one edge from \( E_{\pi/2} \) because all edges of \( E_{\pi/2} \) are incident to \(v_1\), and thus, the set \( \mathcal R \) contains at least one pair \(ab\) with \( \langle a, b\rangle < 0\leq m \). Therefore, the size of the balancing matching for \(\mathcal M_4\) is strictly less than the size of \(\mathcal M_2\), the balancing matching for \(\mathcal M_1\). This contradicts our choice of \( \mathcal M_1 \). Therefore, \(H_{\mathcal M_1}(x)<0\) for some \(x\in \mathbb R^d\), and so, \( \mathcal M_1 \) is the desired matching. \hfill \(\square\)

\begin{remark}
    The reader can easily check that instead of \(H_{\mathcal M} \) in the proof of Theorem~\ref{theorem:open-Tverberg-R-d}, one may use the function \(G_{\mathcal M}:\mathbb R^d \to \mathbb R \) defined by
    \[
        G_{\mathcal M}(x)=\max \left\{ \left\|x-(a+b)/2\right\| -\left\|(a-b)/2\right\|: ab\in \mathcal M \right\}.
    \]
    The structure and key steps of the proof using this function is an almost word-by-word repetition of the above proof.
\end{remark}

Also, we refer to Remark~\ref{remark:Q-vs-open-balls}, where we suggest an alternative approach to finish the proof of Theorem~\ref{theorem:open-Tverberg-R-d}. Since this approach deals with the function \(Q\) introduced in the proof of Theorem~\ref{theorem:blue-red-matching-in-r-d}, we discuss it only at the end of Section~\ref{section:proof-perfmat-color}.

\section{Proof of Theorem~\ref{theorem:blue-red-matching-in-r-d}}
\label{section:proof-perfmat-color}

Let \( R \) and \( B \) be the sets of \( n \) red points and \( n \) blue points in \(\mathbb R^d\), respectively.
Consider the function \( Q \) depending on a red-blue perfect matching \(\mathcal M \) for \( R \cup B \) defined by
\[ 
    Q(\mathcal M) = \sum_{rb\in \mathcal M} \|r - b\|^2 =
    \sum_{ r\in R } \| r \|^2  + \sum_{ b\in B } \|b\|^2- 2\sum_{rb\in \mathcal M} \langle r,b\rangle.
\]

Let \( \mathcal M_1 \) be a matching for which the function \( Q \) attains its maximum.
Next, we consider the function \( H_{\mathcal M_1}: \mathbb R^d \to \mathbb R \) from Subsection~\ref{section:prop-extr-point} defined by
\[
H_{\mathcal M_1}(x) = \max\big \{\langle r-x,b-x\rangle: rb\in \mathcal M_1 \big\}.
\]
If \( H_{\mathcal M_1}(x) < 0\)  for some \( x\in \mathbb R^d \), then this point \(x\) is a common point of the balls \(\Ball (ab)\), where \(ab\in \mathcal M_1\), and so, \( \mathcal M_1 \) is a Tverberg matching.  Thus, without loss of generality, we assume that \( H_{\mathcal M_1}\) attains its minimum at the point \( o \) and \(H_{\mathcal M_1}(o)>0\). Applying Lemma~\ref{lemma:point_in_convex_hull} to the matching \( \mathcal M_1 \), we have that \( o \) lies in the convex hull of the midpoints of pairs from the balancing matching \( \mathcal M_2 \) for \(\mathcal M_1\). Assuming \( \mathcal M_2=\{ r_1b_1,\dots, r_kb_k\} \), where \(r_i\in R \) and \( b_i \in B \), we have
\[
o= \sum_{i=1}^k \lambda_{i} (r_i+b_i),
\]
where \( \lambda_{i}\geq 0 \) and \(\sum_{i=1}^k \lambda_i=1\), and so,
\[
\sum_{i=1}^k \lambda_i r_i=-\sum_{i=1}^k \lambda_i b_i.
\]
This equality yields
\[
0\geq \Big\langle \sum_{i=1}^k \lambda_i r_i,\sum_{i=1}^k \lambda_i b_i \Big\rangle = \sum_{i=1}^k \lambda_i^2 \langle r_i, b_i\rangle + \sum_{1\leq i<j\leq k} \lambda_i \lambda_j \big(\langle r_i, b_j \rangle + \langle r_j, b_i \rangle\big).
\]
Using \( \langle r_i, b_i \rangle = H_{\mathcal M_1}(o) \) and supposing \(\langle r_i, b_j \rangle + \langle r_j, b_i \rangle\geq 2H_{\mathcal M_1}(o)\) for all distinct \( i \) and \( j \), we get
\[
    0\geq H_{\mathcal M_1}(o) \cdot \Big( \sum_{i=1}^m \lambda_i \Big)^2=H_{\mathcal M_1}(o)>0,
\]
a contradiction. Therefore, \(\langle r_i, b_j \rangle + \langle r_j, b_i \rangle< 2H_{\mathcal M_1}(o)= \langle r_i, b_i\rangle +\langle r_j, b_j\rangle \) for some distinct \( i\) and \( j\). 
Next, consider the new perfect red-blue matching 
\[ 
    \mathcal M_3= \mathcal M_1\setminus \{r_ib_i, r_jb_j\} \cup \{r_ib_j, r_jb_i\}
\]
satisfying the inequality 
\[ 
Q(\mathcal M_3)=  \sum_{r\in R} \|r\|^2  + \sum_{b\in B} \|b\|^2- 2\sum_{rb\in \mathcal M_3} \langle r,b\rangle >
\sum_{r\in R} \|r\|^2  + \sum_{b\in B} \|b\|^2- 2\sum_{rb\in \mathcal M_1} \langle r,b\rangle = Q(\mathcal M_1).
\] 
This contradicts our choice of \( \mathcal M_1 \). Therefore, \( H_{ \mathcal M_1} (x) \leq 0 \) for some \(x\in \mathbb R^d\), and so, \( \mathcal M_1 \) is the desired matching.
\hfill \(\square\)
\begin{remark}
\label{remark:function Q in Huemer paper}
    The authors of~\cite{huemer2019matching} consider exactly the same  function \(Q\) in the context of Theorem~\ref{theorem:blue-redmatching}, the two-dimensional version of Theorem~\ref{theorem:blue-red-matching-in-r-d}, and show that a matching for which the function \( Q \) attains its maximum is a Tverberg matching. In particular, they apply Helly's theorem and reduce Theorem~\ref{theorem:blue-redmatching} to the case when there are exactly 3 points of each color.
\end{remark}

\begin{remark}
\label{remark:Q-vs-open-balls}
    The reader can easily check that the matchings \(\mathcal M_1\) and \( \mathcal M_4 \) from the proof of Theorem~\ref{theorem:open-Tverberg-R-d} satisfy the inequality \(Q(\mathcal M_1)<Q(\mathcal M_4)\), where \( Q \) is the function introduced in the proof of Theorem~\ref{theorem:blue-red-matching-in-r-d}. The key reason is that \(\mathcal M_4\) contains at least one edge \(ab\not\in \mathcal M_1\) with \( \langle a, b\rangle < 0\). This implies that one can use the values of the function \(Q\) as an alternative invariant in the proof of Theorem~\ref{theorem:open-Tverberg-R-d}.
\end{remark}

\section{Discussion: Proving Tverberg-type theorems using~the~method~of~infinite~descent}
\label{section:discussion}

First, we sketch the proof of Tverberg's theorem by Roudneff~\cite{Roudneff2001}. For an \(r\)-partition \( \mathcal P \) of a set of \((r-1)(d+1)+1\) points in \(\mathbb R^d\), consider the function \(h_{\mathcal P}:\mathbb R^d\to \mathbb R \) defined by
\[
    h_{\mathcal P}(y)=\sum_{X\in \mathcal P} \dist^2 (y, \conv X),
\]
where \( \dist(A, B) \) is the distance between sets \(A, B\subset \mathbb R^d\). Since the function \( h_{\mathcal P} \) is convex, it attains its minimum. Choose a partition \( \mathcal P_1 \) for which this minimum is the smallest possible. If this minimum is 0, then we are done. Therefore, we suppose that it is positive and attained at \(y=y_{\mathcal P_1}\). Then, analyzing the arrangement of the point \( y_{\mathcal P_1} \) and the convex hulls of \( X \in \mathcal P_1 \), we find a partition \( \mathcal P_2 \) such that 
\[
h_{\mathcal P_2}(y_{\mathcal P_1})<h_{\mathcal P_1} (y_{\mathcal P_1}),
\]
a contradiction. Moreover, the partition \( \mathcal P_2 \) slightly differs from \(\mathcal P_1 \) --- they share all but two common sets. Also, remark that Tverberg and Vre{\'{c}}ica~\cite{Tverberg1993} used a similar approach but instead of \( h_{\mathcal P} \), they consider a different function \(g_{\mathcal P}:\mathbb R^d\to \mathbb R\) defined by 
\[
    g_{\mathcal P}(y)=\max_{X\in \mathcal P} \dist(y, \conv X).
\]

The proofs of our Tverberg-type results, Theorems~\ref{theorem:open-Tverberg-R-d} and~\ref{theorem:blue-red-matching-in-r-d}, are based on the method of infinite descent as well but involve novel ideas. At the last step of the proof of Theorem~\ref{theorem:open-Tverberg-R-d}, we consider a new perfect matching \(\mathcal M_4 \) that can differ from \( \mathcal M_1\) \textit{in more than two edges}, that is, the matchings can be very different (recall that in Roudneff's proof, the partitions \( \mathcal P_1 \) and \( \mathcal P_2 \) differs only in two sets). In the proof of Theorem~\ref{theorem:blue-red-matching-in-r-d}, we consider \textit{two functions} instead of one as in Roudneff's proof: the function \( Q \) depending only on matchings and the function \( H_{\mathcal M_1}:\mathbb R^d\to \mathbb R\) depending on a point in \( \mathbb R^d \). Then there are two possibilities: Either the minimum point \( o \) of the function \( H_{\mathcal M_1} \) is the desired intersection point of balls or this point allows to find a new matching \( \mathcal M_3 \) increasing the value of the function \( Q \). 

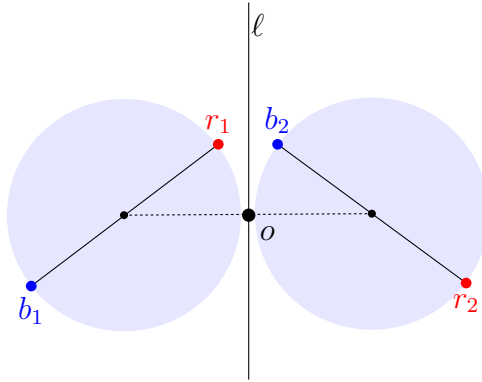
\begin{figure}[h]
  \centering
\begin{tikzpicture}[scale = 2]
\fill [color=blue, opacity=0.1] (8.53,-7.11) circle (0.77cm);
\draw (9.35, -5.7)-- (9.35,-8.2);
\fill [color=blue, opacity=0.1] (10.16,-7.1) circle (0.77cm);
\draw (9.15,-6.64)-- (7.92,-7.58);
\draw (9.54,-6.64)-- (10.78,-7.56);
\draw [dash pattern=on 1pt off 1pt] (8.53,-7.11)-- (10.16,-7.1);
\fill [color=red] (9.15,-6.64) circle (1.0pt);
\fill [color=blue] (7.92,-7.58) circle (1.0pt);
\fill [color=blue] (9.54,-6.64) circle (1.0pt);
\fill [color=red] (10.78,-7.56) circle (1.0pt);
\fill [color=black] (8.53,-7.11) circle (0.75pt);
\fill [color=black] (10.16,-7.1) circle (0.75pt);
\fill [color=black] (9.35,-7.11) circle (1.25pt);
\node at (9.35,-7.11) [below right] {$o$};
\node at (9.15,-6.64) [above] {\textcolor{red}{$r_1$}};
\node at (7.92,-7.58) [below] {\textcolor{blue}{$b_1$}};
\node at (9.54,-6.64) [above] {\textcolor{blue}{$b_2$}};
\node at (10.78,-7.56) [below] {\textcolor{red}{$r_2$}};
\draw[color=black] (9.41, -5.85) node {$\ell$};
\end{tikzpicture}
  \caption{The only possible switching increases the value of $H_{\mathcal M}$ at \(o\).}\label{figure:example}
\end{figure}

Interestingly, a standard argument applied only to the function \( H_{\mathcal M}\) (see Subsection~\ref{section:prop-extr-point}) does not allow us to finish the proof of Theorem~\ref{theorem:blue-red-matching-in-r-d}. To illustrate this, consider the example of two red points \(r_1,r_2\) and two blue points \(b_1,b_2\) drawn\footnote{The reader may assume that the points \(r_1\) and \(b_2\) are symmetric with respect to the line \(\ell\) and the points \(r_2\) and \(b_1\) are symmetric with respect to \(\ell\) as well. Moreover, additionally assume that \(\angle r_1ob_2=\pi/4\) and the distance between the \textit{non-intersecting} disks \( B(r_1b_1)\) and \(B(r_2b_2)\) is close to \(0\).} in Figure~\ref{figure:example}. Clearly, \( \mathcal M_2 =\{r_1b_2, r_2b_1\} \) is the only desired matching. Choosing the second matching \( \mathcal M_1=\{r_1b_1, r_2b_2\} \) as a starting matching, we expect to show the inequality
\(
    H_{\mathcal M_2}(o)< H_{\mathcal M_1}(o),
\)
where \(o\) is the minimum point of \(H_{\mathcal M_1}\). However, it
does not hold because
\[
 H_{\mathcal M_2}(o)=\langle r_1, b_2\rangle > \langle r_1,b_1 \rangle = \langle r_2,b_2\rangle=H_{\mathcal M_1}(o).
\]
This obstacle shows that an extra argument is needed to show that the minimum of the function \(H_{\mathcal M_2} \) is less than \(H_{\mathcal M_1}(o)\). To finish the proof of Theorem~\ref{theorem:blue-red-matching-in-r-d}, we introduce a new function \(Q\) depending only on matching. Unfortunately, we did not find a more direct approach to complete the argument. 

\section{Open problems}
\label{section:open_problems}
\subsection{Intersection of balls}

First, we recall Problem~4.1 from~\cite{soberon2020tverberg}. 

\begin{problem}
    \label{problem:Tverbergcycle}
    Is it true that for any finite set of points in \(\mathbb R^d\), there is a Tverberg cycle?
\end{problem}

One of the possible approaches to answer this problem affirmitevely is to follow the proof of Theorem~\ref{theorem:cycleintheplane}. In particular, we can apply the following higher-dimensional generalization of Observation~\ref{key-observation}: Given \( d \) orthogonal hyperplanes with a common point \( o \) partition \(\mathbb R^d\) into \(2^d\) parts (orthants), the point \( o \) lies in any ball with a diameter whose endpoints are in two opposite orthants. However, the difficulty arises in proving that there are \(d\) orthogonal hyperplanes such that the number of points in the opposite orthants are (almost) the same\footnote{Note that in proof of Theorem~\ref{theorem:cycleintheplane}, we choose the lines \(\ell_{\alpha}\) and \(\ell_{\alpha+\pi/2}\) containing some points of the set properly arranged on these lines. This leads to another difficulty that the desired hyperplanes contain some points of the set such that we may find a proper cycle.}. This statement can be viewed as a discrete vertion of the following question resembling the celebrated Gr\"unbaum hyperplane mass partition problem (see~Subsection~2.1 in~\cite{roldan2022survey}).

\begin{problem} 
Given a probability measure \(\mu\) in \(\mathbb R^d\), absolutely continuous with respect to the Lebesgue measure, does there exist \(d\) orthogonal hyperplanes dividing \(\mathbb R^d\) into \(2^d\) orthants such that the opposite orthants are of the same size with respect to the measure~\(\mu\)?
\end{problem}

One can also consider the blue-red variation of Problem~\ref{problem:Tverbergcycle}: For any set of \( n \) red points and \( n \) blue points in \(\mathbb R^d\), there is a Tverberg red-blue cycle. It is worth mentioning that there is a simple contrexample to this statement even in the plane: The set of the vertices of a rectangle that is not a square colored in red and blue colors alternatively.

Another problem generalizes the main result in~\cite{huemer2019matching}.

\begin{problem}
    \label{problem:Tverberg matching}
    Is true that for any even set of points in \( \mathbb R^d \), the matching maximizing the sum of the distances between points is a (open) Tverberg matching?
\end{problem}

\subsection{Intersection of lenses}

\label{section:discussion-intersection-of-lenses}

Once we prove that there is a Tverberg matching, we can ask whether there is a common point of balls lying relatively deep in all these balls. One of possible ways to deal with this question is to consider \(\alpha\)-lenses instead of balls. There are many interesting results and open problems~\cites{barany1987extension, barany1987covering, magazinov2017positive} related to intersection properties of \(\alpha\)-lenses\footnote{Compare Lemma~3 from~\cite{barany1987extension} and the proofs of the obtuse graph properties in Subsection~\ref{section:prop-obt-graph}.}. For two points \(x,y\in \mathbb R^d\) and \( \alpha \in [0,\pi]\), the \textit{\(\alpha\)-lens} \( \alpha(xy) \) is defined by
\[
    \alpha(xy):=\{x,y\}\cup \{z\in \mathbb R^d\setminus\{x,y\}: \angle xzy\geq \alpha\}.
\]
In particular, if \(\alpha =\pi/2\), then \(\alpha(xy)=\Ball (xy)\).

For an even set \( S \) of points in \( \mathbb R^d \), denote by \( A(S) \) the maximum \( \alpha\in [0, \pi] \) such that there is a matching for the set \( S \) such that the \(\alpha\)-lenses induced by its edges intersect. Let \( A_d \) be the infimum of a set of \( A (S) \) for all even sets \( S \) in \(\mathbb R^d\). Considering the multiset of the vertices of a regular simplex in \(\mathbb R^d\) taken twice, one can show that \( A_d \geq \arccos (-1/d) \). Also, the result of Dumitrescu, Pach, and T\'oth~\cite{dumitrescu2009drawing} implies that \( A_2= 2\pi/3 \); see also~Section 3 in~\cite{soberon2020tverberg}. All these observations lead to the following problem.

\begin{problem}
    Find \(A_d\).
\end{problem}

According to Theorem~\ref{theorem:open-Tverberg-R-d}, for any even set \( S \) of distinct points in \(\mathbb R^d\), we have \( A(S)>\pi/2\). This result can be viewed as the first step towards proving that \( A_d>\pi/2\). Unfortunately, it seems that our approach does not allow to show that. 

\subsection{Intersection of homothets}

Another interesting generalization of Tverberg graphs occurs if in the definition of a Tverberg graph, we replace Euclidean balls by homothets of a centrally symmetric convex body. Remark that the intersection properties of homothets were studied in the recent years; see~\cites{naszodi2017arrangements, polyanskii2017pairwise} and other papers citing them. 

For a centrally symmetric about the origin convex body \( K\subset \mathbb R^d \) and two points \(x,y\in \mathbb R^d\), let \(\BallK (xy):=\frac{x+y}{2}+\lambda K\), where \(\lambda\) is the least positive number such that \(\frac{x+y}{2}+\lambda K\) covers the points \(x\) and \(y\), that is, the segment \(xy\) is a diameter of \(K(xy)\). Let \( G \) be a graph whose vertex set is a finite set of points in \( \mathbb{R}^d \). For a centrally symmetric about the origin convex body \(\BallK\), we say that \( G \) is \textit{a \(K\)-Tverberg graph} if
\[
    \bigcap_{ xy\in E(G)} \BallK(xy) \neq \emptyset.
\]
This definition leads us to the following problem.
\begin{problem}
\label{conjecture:KTverberg-red-blue}
Is it true that for any centrally symmetric convex body \( K \subset \mathbb R^d\) and an even set of points in \(\mathbb R^d\), there is a perfect matching that is a \(K\)-Tverberg graph? Is it true that the matching maximizing the sum of the distances between matched points is a \( K \)-Tverberg graph? Here we assume that the distance between two points is measured in the normed \(d\)-space whose unit ball is \(K\). (Compare with the main result in~\cite{bereg2019maximum}.)
\end{problem}

\section{Data availability}
Data sharing not applicable to this article as no datasets were generated or analysed during the current study.

\bibliographystyle{siam}
\bibliography{biblio}

\end{document}